\documentclass[11pt, letterpaper]{amsart}

\usepackage{fullpage}

\usepackage{setspace}
\usepackage[pdftex]{hyperref}

\usepackage{enumerate}

\usepackage{graphicx}

\usepackage[all,cmtip]{xy}
\usepackage{verbatim}

\usepackage{amsmath}
\usepackage{amssymb}
\usepackage{amsthm}
\usepackage{mathrsfs}

\newcommand{\Q}{\mathbb{Q}}
\newcommand{\N}{\mathbb{N}}
\newcommand{\Z}{\mathbb{Z}}
\newcommand{\C}{\mathbb{C}}

\newcommand{\J}{\mathcal{J}}
\renewcommand{\a}{\mathfrak{a}}

\renewcommand{\O}{\mathcal{O}}

\newcommand{\of}{\circ}     

\DeclareMathOperator{\ord}{{ord}}

\DeclareMathOperator{\Supp}{{Supp}}

\DeclareMathOperator{\Spec}{{Spec}}

\DeclareMathOperator{\Ex}{Ex}
\DeclareMathOperator{\Div}{Div}
\newcommand{\cE}{\check{E}}

\theoremstyle{plain}

\newtheorem*{theorem*}{Theorem}
\newtheorem{theorem}{Theorem}[section]

\newtheorem{lemma}[theorem]{Lemma}

\theoremstyle{remark}

\theoremstyle{definition}

\newtheorem*{question*}{Question}

\newtheorem*{example*}{Example}

\begin{document}
\bibliographystyle{amsalpha}

\title{Integrally Closed Ideals on Log Terminal Surfaces are Multiplier Ideals}
\author{Kevin Tucker}
\address{Department of Mathematics, University of Michigan, Ann Arbor, Michigan  48109}
\email{kevtuck@umich.edu}
\thanks{The author was partially supported by the NSF under grant DMS-0502170.}

\begin{abstract}
We show that all integrally closed ideals on log terminal surfaces are multiplier ideals by extending an existing proof for smooth surfaces.
\end{abstract}

\maketitle
\tableofcontents

\spacing{1.3}

\section{Introduction}
Throughout this article, we will consider a scheme $X = \Spec \O_{X}$, where $\O_{X}$ is a two-dimensional local normal domain essentially of finite type over $\C$.  
Our purpose is to address the following question, raised in \cite{LLS08}:

\begin{question*}
If $X$ has a rational singularity, is every integrally closed ideal a multiplier ideal?
\end{question*}

\noindent
When $X$ is regular, an affirmative answer was given concurrently by \cite{LipWat03} and \cite{FavreJonsson05}.  Our main result is to generalize their methods to prove the following:

\begin{theorem}
\label{theorem}
Suppose $X$ has log terminal singularities.  Then every integrally closed ideal is a multiplier ideal.
\end{theorem}

Recall that log terminal singularities are necessarily rational (see Theorem 5.22 in \cite{KM}).  In a sense, this gives a complete answer to the above question: if $X$ has a rational singularity but fails to be log terminal, every multiplier ideal is strictly proper.  In particular, $O_{X}$ itself is not a multiplier ideal.  

\smallskip

There are several difficulties in trying to extend the techniques used in  \cite{LipWat03}.  One must show that successful choices can be made in the construction (specifically, the choice of $\epsilon$ and $N$ in Lemma 2.2 of \cite{LipWat03}).  Here, it is essential that $X$ has log terminal singularities.
Further problems arise from the failure of unique factorization to hold for integrally closed ideals.  As $X$ is not necessarily factorial, we may no longer reduce to the finite colength case.  In addition, the crucial contradiction argument which concludes the proof in \cite{LipWat03} does not apply.  These nontrivial difficulties are overcome by using a relative numerical decomposition for divisors on a resolution over $X$.  

Our presentation is self-contained and elementary.  Section 2 contains background material covering the relative numerical decomposition, antinef closures, and some computations using generic sequences of blowups.  Section 3 is dedicated to the constructions and arguments in the proof of Theorem~\ref{theorem}.

\section{Background}

\subsection{Relative Numerical Decomposition}  Consider $X = \Spec \O_{X}$, where $\O_{X}$ is a two-dimensional local normal domain essentially of finite type over $\C$. Let $x \in X$ be the unique closed point,
and suppose $f: Y \to X$ is a projective birational morphism such that $Y$ is regular and $f^{-1}(x)$ is a simple normal crossing divisor.  Let $E_{1}, \ldots, E_{u}$ be the irreducible components of $f^{-1}(x)$, and $\Lambda = \oplus_{i} \Z E_{i} \subset \Div(Y)$ the lattice they generate.

\smallskip

The intersection pairing $\Div(Y) \times \Lambda \to \Z$ induces a negative definite $\Q$-bilinear form on $\Lambda_{\Q}$ (see \cite{Artin66} for an elementary proof). Consequently, there is a dual basis $\cE_{1}, \ldots, \cE_{u}$ for $\Lambda_{\Q}$ defined by the property that
\[
\cE_{i} \cdot E_{j} = - \delta_{ij} = \left\{ 
\begin{array}{c@{\quad}l}
-1 & i = j \\
0 & i \neq j
\end{array}
\right..
\]
Recall that a divisor $D \in \Div_{\Q}(Y)$ is said to be $f$-antinef if $D \cdot E_{i} \leq 0$ for all $i = 1, \ldots, u$.  In this case, $D$ is effective if and only if $f_{*} D$ is effective (see Lemma 3.39 in \cite{KM}). In particular, $\cE_{1}, \ldots, \cE_{u}$ are effective.

\smallskip

If $C \in \Div_{\Q}(X)$, we define the numerical pullback of $C$ to be the unique $\Q$-divisor $f^{*}C$ on $Y$ such that $f_{*}f^{*} C = C$ and $f^{*}C \cdot E_{i} = 0$ for all $i =  1, \ldots, u$.  Note that, when $C$ is Cartier or even $\Q$-Cartier, this agrees with the standard pullback of $C$.  If $D \in \Div_{\Q}(Y)$, we have
\begin{equation}
\label{decomp}
D = f^{*}f_{*}D + \sum_{i} (-D \cdot E_{i}) \cE_{i}.
\end{equation}
We shall refer to this as a relative numerical decomposition for $D$.
Note that, even when $D$ is integral, both $f^{*}f_{*} D$ and $\cE_{1}, \ldots, \cE_{u}$ are likely non-integral.  The fact that $f^{*}f_{*} D$ and $\cE_{1}, \ldots, \cE_{u}$ are always integral divisors when $X$ is smooth and $D$ is integral is equivalent to the unique factorization of integrally closed ideals.  See \cite{Lipman69} for further discussion.

\subsection{Antinef Closures and Global Sections}  Suppose now that $D' = \sum_{E} a'_{E}E$ and
$D'' = \sum_{E} a''_{E} E$ are $f$-antinef divisors, where the sums range over the prime divisors 
$E$ on $Y$.  It is easy to check that $D' \wedge D'' = \sum_{E} \min\{a'_{E}, a''_{E}\} E$ is also $f$-antinef.  Further, any integral $D \in \Div(Y)$ is dominated by some integral $f$-antinef divisor (e.g. $(f^{-1})_{*} f_{*} D + M(\cE_{1} + \cdots + \cE_{u})$ for sufficiently large and divisible $M$).  In particular, there is a unique smallest integral $f$-antinef divisor $D^{\sim}$, called the $f$-antinef closure of $D$, such that $D^{\sim} \geq D$. One can verify that $f_{*}D = f_{*}D^{\sim}$, and in addition the following important lemma holds (see Lemma 1.2 of \cite{LipWat03}).  The proof also gives an effective algorithm for computing $f$-antinef closures.

\begin{lemma}
\label{sections}
For any $D \in \Div(Y)$, we have $f_{*} \O_{Y}(-D) = f_{*} \O_{Y}( - D^{\sim})$.
\end{lemma}

\begin{proof}
Let $s_{D} \in \N$ be the sum of the coefficients of $D^{\sim} - D$ when written in terms of $E_{1}, \ldots, E_{n}$.  If $s_{D} = 0$, then $D = D^{\sim}$ is $f$-antinef and the statement follows trivially.  Else, there is an index $i$ such that $D \cdot E_{i} > 0$.  As $E_{i} \cdot E_{j} \geq 0$ for $j \neq i$, we must have
\[
D \leq D + E_{i} \leq D^{\sim} = (D + E_{i})^{\sim}.
\]
Thus, $s_{D + E_{i}} = s_{D} -  1$ and by induction we may assume
\[
f_{*} \O_{Y}(-(D+E_{i}) = f_{*} \O_{Y}(-(D+E_{i})^{\sim}) = f_{*} \O_{Y}(-D^{\sim}) 
\]
and it is enough to show $f_{*} \O_{Y}(-D) = f_{*} \O_{Y}(-(D+ E_{i})$.  Consider the exact sequence
\[  \xymatrix{ 
0 \ar[r] &   \O_{Y}(-(D+E_{i}))   \ar[r] &   \O_{Y}(-D)   \ar[r] &   \O_{E_{i}}(-D)   \ar[r] & 0   .
} \]
Since $\deg(\O_{E_{i}}(-D)) = - D \cdot E_{i} < 0$, we have $f_{*} \O_{E_{i}}(-D) = 0$; applying $f_{*}$ yields the desired result.
\end{proof}

\subsection{Generic Sequences of Blowups}
\label{blowups}
In the proof of Theorem \ref{theorem}, we will make use of the following auxiliary construction.
Suppose $x^{(i)}$ is a closed point of $E_{i}$ with $x^{(i)} \not\in E_{j}$ for $j \neq i$.  A generic sequence of $n$-blowups over $x^{(i)}$ is:
\[  \xymatrix{ 
Y = Y_{0} & Y_{1} \ar[l]_{\quad \sigma_{1}}  &   \ar[l]_{\sigma_{2}} \cdots &  \ar[l]_{\sigma_{n-1}} Y_{n-1} & Y_{n} \ar[l]_{\quad \sigma_{n}}
} \]
where
$\sigma_{1} : Y_{1} \to Y_{0}$ is the blowup of $Y_{0}= Y$ at $x_{1} := x^{(i)}$, and  
$\sigma_{k}: Y_{k} \to Y_{k-1}$ is the blowup of $Y_{k-1}$ at a generic closed point $x_{k}$ of $(\sigma_{k-1})^{-1}( x_{k-1} )$ for $k = 2, \ldots, n$.  Let $\sigma: Y_{n} \to Y$ be the composition $\sigma_{n} \of \cdots \of \sigma_{1}$.  We will denote by $E(1), \ldots, E(u)$ the strict transforms of $E_{1}, \ldots, E_{u}$ on $Y_{n}$.  Also, let $E(i, x^{(i)}, k)$, $k = 1, \ldots, n$, be the strict transforms of the $n$ new $\sigma$-exceptional divisors created by the blowups $\sigma_{1}, \ldots, \sigma_{n}$, respectively.

\begin{lemma}
\label{gen}
\begin{enumerate}[(a.)]
\item
Let $\sigma: Y_{n} \to Y$ be a generic sequence of blowups over $x^{(i)} \in E_{i}$.  Then one has
\[
\cE(i) \leq \cE(i,x^{(i)},1) \leq \cdots \leq \cE(i,x^{(i)},n).
\]
\item
\label{genb}
Suppose $D \in \Div(Y)$ is an integral $(f \of \sigma)$-antinef divisor such that $E_{i}$ is the unique component of $\sigma_{*}D$ containing $x^{(i)}$.
If $\ord_{E(i)}D = a_{0}$ and $\ord_{E(i, x^{(i)}, k)}D = a_{k}$ for $k = 1, \ldots, n$, then
\[
a_{0} \leq a_{1} \leq \cdots \leq a_{n}.
\]
Further, $a_{0} < a_{n}$ if and only if
\[
\left( \sum_{k=1}^{n} (-D \cdot E(i, x^{(i)}, k)) \cE(i, x^{(i)}, k)  \right) \geq \cE(i).
\]
\end{enumerate}
\end{lemma}

\begin{proof}
If $n = 1$, we have 
\[
\cE(i,x^{(i)},1) = \left( \sigma^{*}\cE_{i} + E(i,x^{(i)},1) \right) \geq \sigma^{*}\cE_{i} = \cE(i)
\]
\[
D = \sigma^{*}\sigma_{*}D + (-D\cdot E(i,x^{(i)},1)) \cE(i,x^{(i)},1).
\]
The general case of both statments follows easily by induction.
\end{proof}

\section{Main Theorem}

\subsection{Log Terminal Singularities and Multiplier Ideals}
Once more, suppose $x \in X$ is the unique closed point and $f:Y \to X$ is a projective birational morphism such that $Y$ is regular and $f^{-1}(x)$ is a simple normal crossing divisor.  Let $E_{1}, \ldots, E_{u}$ be the the irreducible components of $f^{-1}(x)$, and
let $K_{Y}$ be a canonical divisor on $Y$.  Then $K_{X} := f_{*}K_{Y}$ is a canonical divisor on $X$.  If we write the relative canonical divisor as
\[
K_{f} := K_{Y} - f^{*}K_{X} = \sum_{i} b_{i} E_{i}
\]
then $X$ has log terminal singularities if and only if $b_{i} > -1$ for all $i = 1, \ldots, u$.  In this case, $X$ is automatically $\Q$-factorial (see  Proposition 4.11 in \cite{KM}, as well as  \cite{deFernexHacon08} for recent developments).

\smallskip

If $\a \subseteq \O$ is an ideal, recall that $f:Y \to X$ as above is said to be a log resolution of $\a$ if $\a \O_{Y} = \O_{Y}(-G)$ for an effective divisor $G$ such that $\Ex(f) \cup \Supp(G)$ has simple normal crossings.  In this case, we can define the multiplier ideal of $(X, \a)$ with coefficient $\lambda \in \Q_{>0}$ as
\[
\J(X, \a^{\lambda}) = f_{*}\O_{Y}(\lceil K_{f} - \lambda G \rceil ) .
\]
See \cite{Tucker07-1} for an introduction in a similar setting, or \cite{Lazarsfeld04} for a more comprehensive overview.  Also recall that $\a$ is integrally closed if and only if
\[
\a = f_{*}\O_{Y}(-G).
\]

\subsection{Choosing $\a$ and $\lambda$}
We now begin the proof of Theorem~\ref{theorem}.  For the remainder, assume $X$ is log terminal, and let $I \subseteq \O_{X}$ be an integrally closed ideal.  In this section, we construct another ideal $\a \subseteq \O_{X}$ along with a coefficient $\lambda \in \Q_{>0}$, and in the following section it will be shown that $\J(X, \a^{\lambda}) = I$.  Let $f:Y \to X$ a log resolution of $I$ with exceptional divisors $E_{1}, \ldots, E_{u}$.  Suppose $I \O_{Y} = \O_{Y}(-F^{0})$, and write
\[
K_{f} = \sum_{i=1}^{u} b_{i} E_{i}
\]
\[
F^{0} = (f^{-1})_{*}f_{*}(F^{0}) + \sum_{i=1}^{u} a_{i}E_{i}.
\]

\smallskip

Choose $0 < \epsilon < 1/2$ such that $\lfloor \epsilon (f^{-1})_{*}f_{*}(F^{0} )\rfloor = 0$ and
\[
\epsilon (a_{i} + 1) < 1 + b_{i}
\]
for $i = 1, \ldots, u$.  Note that, since $X$ is log terminal, $1 + b_{i} > 0$ and any sufficiently small $\epsilon > 0$ will do.  Let $n_{i} := \lfloor \frac{1 + b_{i}}{\epsilon} - (a_{i} + 1) \rfloor \geq 0$, and $e_{i} := (-F^{0} \cdot E_{i})$.  Choose $e_{i}$ closed points $x_{1}^{(i)}, \ldots, x_{e_{i}}^{(i)}$ on $E_{i}$ such that $x_{j}^{(i)} \not\in \Supp \left( (f^{-1})_{*}f_{*}(F^{0}) \right)$ and $x_{j}^{(i)} \not\in E_{l}$ for $l \neq i$.  Denote by $g:Z \to Y$ the composition of $n_{i}$ generic blowups at each of the points $x_{j}^{(i)}$ for $j = 1, \ldots, e_{i}$ and $i = 1, \ldots, u$.  As in Section~\ref{blowups}, denote by $E(1), \ldots, E(u)$ the strict transforms of $E_{1}, \ldots, E_{u}$, and $E(i, x_{j}^{(i)}, 1), \ldots, E(i, x_{j}^{(i)}, n_{i})$ the strict transforms of the $n_{i}$ exceptional divisors over $x_{j}^{(i)}$.

\smallskip

Let $h := f \of g$, $F = g^{*}(F^{0})$, and choose an effective $h$-exceptional integral divisor $A$ on $Z$ such that $-A$ is $h$-ample.  It is easy to see that
\[
K_{g} = \sum_{i=1}^{u} \sum_{j=1}^{e_{i}} \sum_{k = 1}^{n_{i}} k \, E(i, x_{j}^{(i)}, k)
\]
and one checks
\[
K_{g} \cdot E(i) = e_{i} \qquad K_{g} \cdot E(i, x_{j}^{(i)}, k) = \left\{
\begin{array}{c@{\quad}l}
0 & k \neq n_{i} \\
-1 & k = n_{i}
\end{array}
\right. .
\]
It follows immediately that $F + K_{g}$ is $h$-antinef.
Choose $\mu > 0$ sufficiently small that 
\begin{equation}
\label{star}
\lfloor (1 + \epsilon) (F + K_{g} + \mu A) - K_{h} \rfloor =
\lfloor (1 + \epsilon)(F + K_{g}) - K_{h} \rfloor .
\end{equation}
As $-(F + K_{g} + \mu A)$ is $h$-ample, there exists $N > > 0$ such that $G:= N(F + K_{g} + \mu A)$ is integral and $-G$ is relatively globally generated.\footnote{As $X$ is log terminal, it also has rational singularities, and by Theorem 12.1 of \cite{Lipman69} it follows that $-(F + K_{g})$ is already globally generated without the addition of $-A$.  However, the above approach seems more elementary, and avoids unnecessary reference to these nontrivial results.}  In other words, $\a := h_{*} \O_{Z}(-G)$ is an integrally closed ideal such that $\a \O_{Z} = \O_{Z}(-G)$.  Set $\lambda = \frac{1 + \epsilon}{N}$.

\subsection{Conclusion of Proof}
Here, we will show $\J(X, \a^{\lambda}) = I = h_{*} \O_{Z}(-F)$.  Since
\[
\J(X, \a^{\lambda}) = h_{*} \O_{Z}(\lceil K_{h} - \lambda G \rceil) = h_{*} \O_{Z}( - \lfloor \lambda G - K_{h} \rfloor) ,
\]
by Lemma~\ref{sections}, it suffices to show
$
F' := \lfloor \lambda G - K_{h} \rfloor^{\sim} = F$.  In particular, we have reduced to showing a purely numerical statement.

\begin{lemma}
\label{chosenwell}
We have
$F' \leq F$ and $   h_{*} F' = h_{*} F  $.  In addition, for $i = 1, \ldots, u$ and $j= 1, \ldots, e_{i}$,
\[
 \ord_{E(i, x_{j}^{(i)}, n_{i})}(F') =
\ord_{E(i, x_{j}^{(i)}, n_{i})}(F) = \ord_{E(i)}(F).
\]
\end{lemma}

\begin{proof}
Since $F' = \lfloor \lambda G - K_{h} \rfloor^{\sim}$ and $F$ is $h$-antinef ($-F$ is relatively globally generated), it suffices to show these statements with $\lfloor \lambda G - K_{h} \rfloor$ in place of $F'$.
By \eqref{star}, we have
\begin{eqnarray*}
\lfloor \lambda G - K_{h} \rfloor & = & \lfloor (1 + \epsilon)(F + K_{g}) - K_{h} \rfloor \\
& = & F +  \lfloor \epsilon(F + K_{g}) - g^{*}K_{f} \rfloor.
\end{eqnarray*}
Since $\lfloor \epsilon (f^{-1})_{*} f_{*} F^{0} \rfloor = 0$, it follows immediately that $h_{*} \lfloor \lambda G - K_{h} \rfloor = h_{*}F$.  For the remaining two statements, consider the coefficients of $\epsilon(F+K_{g}) - g^{*}K_{f}$.  Along $E(i)$, we have $\epsilon a_{i} - b_{i}$, which is less than one by choice of $\epsilon$.  Along $E(i, x_{j}^{(i)}, k)$, we have $\epsilon(a_{i} + k) - b_{i}$.  This expression is greatest when $k = n_{i}$, where our choice of $n_{i}$ guarantees
\[
\frac{b_{i}}{\epsilon} - a_{i} \leq n_{i} < \frac{b_{i} + 1}{\epsilon} - a_{i}
\]
\[
0 \leq \epsilon(a_{i} + n_{i}) - b_{i} < 1 .
\]
It follows that $\lfloor \lambda G - K_{h} \rfloor \leq F$, with equality along $E(i, x_{j}^{(i)}, n_{i})$.  
\end{proof}

\begin{lemma}
For each $i = 1, \ldots, u$,
\[
(-F' \cdot E(i)) \cE(i) + \sum_{j=1}^{e_{i}}\sum_{k=1}^{n_{i}} (-F' \cdot E(i, x_{j}^{(i)}, k)) \cE(i, x_{j}^{(i)},k) \quad \geq \quad
(-F \cdot E(i))\cE(i).
\]
\end{lemma}

\begin{proof}
If $\ord_{E(i)}F' = \ord_{E(i)}F$, as $F' \leq F$ we have $F' \cdot E(i) \leq F \cdot E(i)$ and the conclusion follows as $\cE(i)$ and $\cE(i, x_{j}^{(i)}, k )$ are effective and $F'$ is $h$-antinef.  Otherwise, if $\ord_{E(i)}F' < \ord_{E(i)}F = \ord_{E(i, x_{j}^{(i)}, n_{i})}F'$, then for each $j = 1, \ldots, e_{i}$ we saw in Lemma~\ref{gen}\eqref{genb} that
\[
\sum_{k=1}^{n_{i}} (-F' \cdot E(i, x_{j}^{(i)}, k)) \cE(i, x_{j}^{(i)},k) \quad \geq \quad \cE(i).
\]
Summing over all $j$ gives the desired conclusion.
\end{proof}

We now finish the proof by showing that $F' \geq F$.  Using the relative numerical decomposition \eqref{decomp} and the previous two Lemmas, we compute
\begin{eqnarray*}
F' & =&  h^{*} h_{*} F' + \sum_{i=1}^{u}(-F' \cdot E(i)) \cE(i) + \sum_{i=1}^{u}\sum_{j=1}^{e_{i}}\sum_{k=1}^{n_{i}} (-F' \cdot E(i, x_{j}^{(i)}, k)) \cE(i, x_{j}^{(i)},k) \\
&=& h^{*}(h_{*} F ) + \sum_{i=1}^{u}\left( (-F' \cdot E(i)) \cE(i) + \sum_{j=1}^{e_{i}}\sum_{k=1}^{n_{i}} (-F' \cdot E(i, x_{j}^{(i)}, k)) \cE(i, x_{j}^{(i)},k) \right) \\
& \geq &  h^{*} h_{*} F + \sum_{i=1}^{u} (-F \cdot E(i)) \cE(i) = F.
\end{eqnarray*}
This concludes the proof of Theorem~\ref{theorem}.

\bibliography{/Users/kevintucker/Desktop/TeX/math.bib}
\end{document}